%% file: master.tex
\documentclass[12pt,a4paper,pagewise]{article}
\usepackage[margin=2cm]{geometry}
\usepackage[utf8]{inputenc}
\usepackage[T1]{fontenc}
\usepackage{amsfonts}
\usepackage{amsfonts,amsmath,amssymb,mathabx}
\usepackage{xcolor}
\usepackage[mathscr]{eucal}
\usepackage{amscd}

\usepackage[scaled=0.90]{helvet} 
\usepackage[english]{babel}
\usepackage[T1]{fontenc}
\usepackage[utf8]{inputenc}
\usepackage{amsmath}
\usepackage{graphicx}
\usepackage{geometry}
\usepackage{amsfonts}
\usepackage{amssymb}
\usepackage{amsthm, thmtools}
\usepackage{xcolor}
\usepackage{mathrsfs}
\usepackage{titling}
\usepackage{abstract}
\usepackage{mathrsfs}
\usepackage{mathtools}
\usepackage{pythontex} 
\usepackage{csquotes}
\usepackage{titlesec}

\newcommand*{\N}{\mathbb{N}}
\newcommand*{\Z}{\mathbb{Z}}
\newcommand*{\Q}{\mathbb{Q}}
\newcommand*{\R}{\mathbb{R}}
\newcommand*{\C}{\mathbb{C}}

\newcommand{\su}{\mathrm{SU}}

\newcommand{\tr}{\mathrm{Tr}}
\newcommand{\vol}{\mathrm{vol}}
\newcommand{\ceq}{\coloneqq}

\numberwithin{equation}{section}

\newtheorem*{theorem*}{Theorem}
\newtheorem*{definition*}{Definition}
\newtheorem*{corollary*}{Corollary}
\newtheorem{thm}{Theorem}[section]

\newtheorem{cor}[thm]{Corollary}
\theoremstyle{definition}
\newtheorem{de}[thm]{Definition}

\newtheorem{rmk}[thm]{Remark}

\usepackage{xpatch}
\usepackage{blindtext}
\makeatletter
\xpatchcmd{\titlepage}{\@restonecolfalse\newpage}{\@restonecolfalse}{}{}
\xpatchcmd{\endtitlepage}{\if@restonecol\twocolumn \else \newpage \fi}{\if@restonecol\twocolumn \else  \fi}{\typeout{success}}{\typeout{fail}}
\makeatother

\def\Z{\mathbb{Z}}

\title{\large Explicit Multiplicities in the Cuspidal Spectrum of $\mathrm{SU}(n,1)$}
\date{}

\usepackage{hyperref}
\hypersetup{
    colorlinks,
    citecolor=black,
    filecolor=black,
    linkcolor=black,
    urlcolor=black
}
\hypersetup{linktocpage}

\begin{document}

\numberwithin{thm}{section}
\numberwithin{de}{section}
\numberwithin{lemma}{section}
\numberwithin{cor}{section}
\numberwithin{ex}{section}
\numberwithin{rmk}{section}

\author{Alexander Stadler}

\maketitle

\begin{abstract}
This paper investigates the cuspidal spectrum of the quotient of the real Lie group $G= SU(n,1)$ and a principal congruence subgroup $\Gamma(m)$ for $m\geq 3$, focusing on the multiplicities of integrable discrete series representations. Using the Selberg trace formula, we derive an explicit formula for the multiplicity $m(\Gamma(m), \pi_\tau)$ of a representation $\pi_\tau$ of integrable discrete series of $G$ within $L^2(\Gamma(m) \backslash G)$. The formula involves the Harish-Chandra parameter $\tau$, the discriminant $D_\ell$ of the imaginary quadratic field $\ell$ over which $G$ is defined and special values of the Dirichlet $L$-function $L_\ell$ associated to $\ell$. We apply these results on the one hand to compute the cuspidal cohomology of locally symmetric spaces $\Gamma(m) \backslash G / K$, where $K$ is a maximal compact subgroup of $G$. On the other hand we use them to reprove a known rationality result involving the values of $L_\ell$ at odd positive integers and make them more explicit. This work extends previous studies on real and quaternionic hyperbolic spaces to the complex hyperbolic case, contributing to the understanding of the spectrum of $\R$-rank one algebraic groups.
\end{abstract}


\pagenumbering{roman}

\pagestyle{headings}
\pagenumbering{arabic}

\include{0Introduction}

\include{1Notation}
\include{2Cuspidalmultiplicities}

\include{lib}

\end{document}

%% file: 0Introduction.tex
\section{Introduction}

Let $G$ be a connected semisimple Lie group and let $\Gamma$ be an arithmetic subgroup of $G$. As in \cite{warner}, we denote by $L_{\Gamma\backslash G}$ the left regular representation on the space $L^2(\Gamma\backslash G)$. It decomposes as 
$$
L^2(\Gamma\backslash G)=L_d^2(\Gamma\backslash G) \oplus L_c^2(\Gamma\backslash G)
$$
into the discrete and continuous spectrum. We study the first summand and denote by $\hat{G}$ the set of unitary equivalence classes of irreducible unitary representations of $G$. We get the direct Hilbert sum (actually the discrete spectrum is defined like this)
$$
L_{\Gamma\backslash G} |_{L_d^2(\Gamma\backslash G)} = \bigoplus_{\pi\in\hat{G}} \pi^{m_d(\Gamma,\pi)},
$$
where $m_d(\Gamma,\pi)$ is the finite multiplicity of $\pi$ in $L_{\Gamma\backslash G} \big|_{L_d^2(\Gamma\backslash G)} $ (strictly speaking it is the multiplicity of the unitary equivalence class $[\pi]$ of irreducible unitary representations, but we simply write $\pi$ for the equivalence class). 
In a certain setting, we will compute these multiplicities explicitly for representations $\pi$ of integrable discrete series.
The Selberg trace formula (or more accurately the special case in \cite{warner}) gives a tool to do this for $\R$-rank one algebraic groups (\cite{warner1}). There are three cases of particular geometric interest:
\begin{enumerate}
    \item $G=SO(n,1)^{0}$, corresponding to the real hyperbolic space
    \item $G=SU(n,1)$, corresponding to the complex hyperbolic space
    \item $G=Sp(n,1)$, corresponding to the quaternionic hyperbolic space
\end{enumerate}
Aspects of the cases 1. and 3. have been studied in \cite{rohlfs} and \cite{Harald} respectively. Case 2. is the topic of this article:
Let $\ell\ceq\mathbb{Q}(\sqrt{-k})$ be an imaginary quadratic extension of $\Q$, and $\mathrm{G}=\mathrm{SU}(n,1)$ the group scheme corresponding to a hermitian form of signature $(n,1)$ over $\ell^{n+1}$.
We will define the arithmetic subgroups $\Gamma$ of $G$ in Section \ref{arithmetic-subgroup} and its congruence subgroups $\Gamma(m)$ in Definition \ref{congruence-subgroup}.
Let $L_\ell(s)$ be the $L$-function attached to the quadratic field $\ell$, i.e. if $\zeta_\ell(s)$ is the Dedekind zeta-function of $\ell$, then
\begin{align*}
    \zeta_\ell(s)=L_\ell(s) \zeta(s).
\end{align*}
To shorten some notation, we will set for an integer $k\geq 2$:
\begin{align*}
    Z_{\ell}(k)\ceq \left \{ \begin{matrix}
        \zeta(k) & \text{if $k$ is even} \\
         L_{\ell}(k) & \text{if $k$ is odd}.
    \end{matrix} \right.
\end{align*}
Let $D_\ell$ be the discriminant of $\ell$ over $\Q$. That is 
    $$
        D_\ell=\left\{ \begin{matrix}
        -k \quad &k\equiv 3 \mod{4} \\
        -4k \quad &else.
        \end{matrix} \right.
    $$
In Section \ref{harish}, we define a system of positive roots for $G$ and attach to a Harish-Chandra parameter $\tau=(\tau_1,...,\tau_n)$ a discrete series representation $\pi_\tau$.
The following, Theorem \ref{mult}, is our main result:

\begin{theorem*}
Let $\Gamma$ be a principal arithmetic lattice of $G=SU(n,1)$ and let $\Gamma(m)$ be a net principal congruence subgroup (as in Definition \ref{congruence-subgroup}) of index $h_m$ in $\Gamma$, with $n\geq 3$ and $m\geq 3$. Let $\pi_\tau \in \hat{G}$ be of integrable discrete series with Harish-Chandra parameter $\tau=(\tau_1,.\ldots,\tau_n)_{\{\varepsilon_i\}}$ with $\tau_{n+1}\ceq - \sum_{i=1}^n \tau_i$, where
\begin{align*}
        \tau_1 &> \tau_2 > \ldots >\tau_n > \tau_{n+1}+n.
        \end{align*}
Then we have
\begin{align*}
    & m(\Gamma(m),\pi_\tau) =d_\tau \mathrm{vol}(\Gamma  \backslash G)h_m \\
    &= \frac{h_m}{4^n} \prod_{i=1}^n  \frac{\left|\prod_{j=i+1}^{n+1} (\tau_i-\tau_j)\right|}{(2\pi)^{i+1}} 
    \cdot |D_\ell|^{\frac{(n-1)(n+2)}{4}}  \cdot \prod_{k=2}^{n+1} Z_{\ell}(k) 
    \cdot \prod_{p \in T_\ell}\lambda_p 
\end{align*}
if $n$ is odd and 
\begin{align*}
& m(\Gamma(m),\pi_\tau) =d_\tau \mathrm{vol}(\Gamma  \backslash G)h_m + O\left( \frac{h_m}{m^n}\right) \\
&= h_m \left( \frac{1}{4^n} \prod_{i=1}^n  \frac{\left|\prod_{j=i+1}^{n+1} (\tau_i-\tau_j)\right|}{(2\pi)^{i+1}} \cdot |D_\ell|^{\frac{n(n+3)}{4}}  \cdot \prod_{k=2}^{n+1} Z_{\ell}(k)  + O\left( \frac{1}{m^n}\right)\right)
\end{align*}
if $n$ is even.
The finite set $T_\ell$ consists of those primes $p$ where $\mathrm{G}(\Q_p)$ is not quasi-split. The rational coefficients $\lambda_p$ are defined in \eqref{lambda}.
\end{theorem*}

As the first application of this multiplicity formula, let $K$ be a maximal compact subgroup of $G$ and let $X\coloneqq G/K$ be the Riemannian symmetric space associated to $G$. As we choose $m \geq 3$, the arithmetic subgroup $\Gamma(m)$ of $G$ is torsion free. Then 
$$
S(\Gamma(m))\coloneqq\Gamma(m) \backslash X
$$
is a locally symmetric space and a smooth  manifold of finite volume. 
Let $E$ be an irreducible, finite dimensional representation of $G$.
We denote by $\Tilde{E}$ the locally constant sheaf on $S(\Gamma(m))$ induced by $E$. 
Recalling the definition of cuspidal cohomology (\cite[1.2]{Harald}), we have 
$$
H_{cusp}^* (S(\Gamma(m)),\Tilde{E})\coloneqq H^*(\mathfrak{g},K,L_{cusp}^2(\Gamma(m) \backslash G)^\infty \otimes E).
$$
Here $L_{cusp}^2(\Gamma(m) \backslash G)$ denotes the space of square integrable functions $f:\Gamma(m) \backslash G \rightarrow \C$ that are cuspidal, i.e.
$$
\int_{\Gamma(m) \cap N \backslash N} f(ng)\,\mathrm{d}n =0 \qquad \text{for almost all } g\in G,
$$
where $N$ denotes the unipotent radical of the unique proper $\Q$-parabolic subgroup of $G$ and the cohomology on the right hand side is the $(\mathfrak{g},K)$-\textit{cohomology} (\cite[Section I.1.2]{borel}).
Now by \cite[Proposition 5.6]{laplacian} we have
$$
H_{cusp}^* (S(\Gamma(m)),\Tilde{E})\coloneqq \bigoplus _{\pi \in \hat{G}_{coh}} H^*(\mathfrak{g},K,\pi_{(K)} \otimes E)^{m(\Gamma(m),\pi)},
$$
where $\hat{G}_{coh}$ is well known by the work of Vogan and Zuckerman. The missing piece in trying to understand the cuspidal cohomology are thus the multiplicities $m(\Gamma,\pi)$. We will apply our results to this situation in Section \ref{cuspcoh}, to get a growth estimate for the dimension of this cohomology.

Recall that we denote by $D_\ell$ the discriminant of the quadratic extension $\ell$ and by $L_\ell$ the $L$-function attached to $\ell$. In Section \ref{rationality} we will prove the following as a consequence of the multiplicity formula above, which is Corollary \ref{rational}:

\begin{corollary*}
Let $\ell=\Q(\sqrt{-k})$ be an imaginary quadratic quadratic field. Let $\Gamma^{(1)}$ and $\Gamma^{(2)}$ be the arithmetic subgroups as in Definition \ref{defgamma} of $SU(2n-1,1)$ and $SU(2n+1,1)$ respectively and let $\Gamma^{(i)}(m^{(i)})$ be net principal congruence subgroups (as in Definition \ref{congruence-subgroup}) of index $h_{m^{(i)}}$ in $\Gamma^{(i)}$, with $n\geq 2$ and $m^{(i)}\geq 3$. Let $\pi_{\tau^{(i)}} \in \hat{G}$ be the discrete series representation with Harish-Chandra parameter $\tau^{(i)}$, such that 
\begin{align*}
        \tau_i^{(1)} &= \tau_i^{(2)} \ for\ 1\leq i \leq 2n-1.,  \\
        \tau_1^{(1)} &> \tau_2^{(1)} > \ldots \tau_{2n-1}^{(1)}>\tau_{2n}^{(1)} +2n -1,\\
        \tau_{2n-1}^{(2)} &> \tau_{2n}^{(2)}> \tau_{2n+1}^{(2)} >\tau_{2n+2}^{(2)}+2n+1, 
        \end{align*}
where we set 
\begin{align*}
    \tau_{2n}^{(1)} &\ceq -\sum_{i=1}^{2n-1} \tau_i^{(1)}, \\
    \tau_{2n+2}^{(2)} &\ceq -\sum_{i=1}^{2n+1} \tau_i^{(2)}.
\end{align*}

This implies that $\pi_{\tau^{(i)}}$ is of integrable discrete series. Then
    \begin{align*}
    &\sqrt{|D_\ell|} \frac{L_\ell(2n+1)}{(2\pi)^{2n+1}} \\
      = &\frac{m(\Gamma^{(2)}(m^{(2)}),\pi_{\tau^{(2)}})}{m(\Gamma^{(1)}(m^{(1)}),\pi_{\tau^{(1)}})} \cdot \frac{32h_{m^{(1)}}}{h_{m^{(2)}}}  \frac{\displaystyle\prod_{\substack{1 \leq i <  2n \\ }} (\tau_i^{(1)}-\tau_{2n}^{(1)})}{\displaystyle \prod_{\substack{1 \leq i < j \leq 2n+2 \\ 2n \leq j}}(\tau_i^{(2)}-\tau_j^{(2)})}\\
      & \cdot \prod_{p \in T_\ell^{(1)}} \lambda_p^{(1)} \cdot \prod_{p \in T_\ell^{(2)}}\frac{1}{\lambda_p^{(2)}} |D_\ell|^{2n}  \frac{(2n+2)!}{(-1)^n B_{2n+2}},
    \end{align*}
    where the sets $T_\ell^{(1)}$ and $T_\ell^{(2)}$ contain the primes at which the corresponding special unitary group is not quasi-split. 
In particular, this implies that $\sqrt{|D_\ell|} \frac{L_\ell(2n+1)}{(2\pi)^{2n+1}}$ is a rational number.
\end{corollary*}
This reflects the rationality result of Siegel in the case of a totally real field $\ell$ of degree $d$:
\begin{align*}
    \frac{\zeta_{\ell}(2n)}{(2\pi i )^{2nd} D_\ell^{\frac{4n-1}{2}}} \in \Q,
\end{align*}
where $\zeta_{\ell}$ denotes the Dedekind zeta function of $\ell$ (\cite[Corollary 3]{hida}). It follows from the functional equation for $L_\ell$ together with the fact that 
\begin{align*}
    L(1-n,\chi_\ell) \in \Q,\ for\ n\geq 1
\end{align*}
(here $\chi_\ell$ denotes the Kronecker-symbol $\left( \frac{D_\ell}{\cdot} \right)$). For imaginary quadratic fields $\ell$, the same is true due to \cite[Theorem 4.2]{wash}, which implies that $\sqrt{|D_\ell|} \frac{L_\ell(2n+1)}{(2\pi)^{2n+1}}$ is rational. However as a novelty, the above formula is explicit and relies only on the integers $m(\Gamma(m^{(i)}),\pi_{\tau^{(i)}})$ and $h_{m^{(i)}}$.

Let us give a short outline of how this article is structured: In Section \ref{gen}, we explain some general features of the special unitary group under consideration, describe the Harish Chandra parametrization of discrete series representations and present how the Selberg trace formula gives the desired multiplicities in this case. In Section \ref{prasad}, we apply Prasad's volume formula to compute $\mathrm{vol}(\Gamma \backslash G)$, which appears in the multiplicity formula. In Section \ref{error} we treat the remaining terms by giving a growth estimate on them. Sections \ref{cuspcoh} and \ref{rationality} contain the applications to cuspidal cohomology and the rationality of values of L-functions respectively, that we mentioned above.

\textbf{Acknowledgment.} The author would like to thank Harald Grobner for his guidance and support throughout this project. Moreover, he is grateful to Johannes Droschl for his helpful comments.

%% file: 1Notation.tex
\section{Notation and Preliminaries}

Let $k$ be a square-free positive integer. For the imaginary quadratic extension $\ell=\Q(\sqrt{-k})$ of $\Q$ we write $\overline{\cdot}$ for the non-trivial automorphism of $\ell$.  We denote by $D_\ell$ the discriminant of $\ell$, given by
\begin{align*}
    D_\ell \ceq \left\{  \begin{matrix}
        -k & if\ k \equiv 3 \mod{4} \\
        -4k & \text{else}.
    \end{matrix} \right.
\end{align*}
For a non-split prime $\mathfrak{p}$ of $l$, the localization $l_{\mathfrak{p}}$ is a quadratic extension of $\Q_p$ and we simply write $l_{(p)}=l_p$. We use $L$ to denote a general quadratic extension of $\Q_p$. \\
If $X$ is a matrix with entries in $\C$, we denote by $X^t$ the transpose of $X$ and by $\overline{X}$ the matrix where the automorphism $\overline{\cdot}$ is applied to every entry. Finally, we denote by $X^\dagger \coloneqq\overline{X^t}$ the conjugate transpose. \\
We denote the set of $m\times n$ Matrices with entries in the ring $R$ by $M_{m\times n}(R)$ and specify $M_n(R)\coloneqq M_{n \times n}(R)$. For the $n \times n$ identity matrix we write $I_n$.
If $B_1,...,B_k$ are quadratic matrices, 
$$
\mathrm{diag}(B_1,...,B_k)
$$
denotes the block diagonal matrix with blocks $B_i$. We denote by $E_{ij}$ the elementary matrix 
$$
\left( \begin{matrix}
    0 &  \ldots  & \ldots & 0 \\
    \vdots & \ddots & 1 & \vdots \\
    0 & \ldots & \ldots & 0
\end{matrix}\right),
$$
which has the unique non-zero entry $1$ at row $i$ and column $j$. When we talk about algebraic groups, we denote them as $\mathrm{G}$, $\mathrm{SL}_n$, $\mathrm{SU}(n,1)$ and so on. Over a field $F$, we denote the $F$-points of $\mathrm{G}$ by $\mathrm{G}(F)$. For instance, over non-archimedian local fields $\Q_p$, we write $\mathrm{G}(\Q_p)$ for the $\Q_p$-points. When we consider the real points of such an algebraic group, we obtain a real Lie group and simply write $G=\mathrm{G}(\R)$, $SL_n=\mathrm{SL}_n(\R)$, $SU(n,1)=\mathrm{SU}(n,1)(\R)$. 

%% file: 2Cuspidalmultiplicities.tex
 
\section{Multiplicities in the Cuspidal Spectrum} \label{gen}

\subsection{Special Unitary Groups}

Consider an imaginary quadratic field extension $\ell\ceq\mathbb{Q}(\sqrt{-k})$ of $\Q$, where $k\in \Z_{>0}$ is square-free. We endow the vector space $V=\ell^{n+1}$ with a non-degenerate hermitian form $h:V \rightarrow V$. When we define
$$
S_h\ceq(h(x_i,x_j))_{(ij)}
$$
for a basis $x_1,...,x_n$ of $V$, the special unitary group corresponding to $h$ 
is the algebraic group given by
$$
\su (S_h)(\Q)\ceq\{X\in M_{n+1} (\ell) |\ X^\dagger S_h X =S_h,\  \det(X)=1\}.
$$
In this article, we will be interested in the hermitian form given by 
$$
S_{n,1}\coloneqq \left(\begin{matrix}
I_n  & 0\\
0  & -1
\end{matrix}\right)
$$
and also denote the corresponding special unitary group $\su (S_{n,1})$ by $\su (n,1)$.
The paramters $n$ and $1$ indicate that the matrix $S$ represents a hermitian form of signature $(n,1)$. More generally, we denote by 
$$
\su (n,k)\ceq \su(S_{n,k}),
$$
where
$$
S_{n,k}=\left(\begin{matrix}
I_n & 0 \\
0 & -I_k
\end{matrix}\right),
$$
with $n,k \geq 0$. The group $\su(n,k)$ is a real form of the group $\mathrm{SL}(n+k)$ and is a simply connected (\cite{milne} 21.43), almost simple group. The special unitary groups always have a unique (up to isomorphism) quasi-split inner form. They are the two groups $\su(n,n)$ and $\su(n+1,n)$: \\
 If $n+k$ is even, the group $\su\left(\frac{n+k}{2},\frac{n+k}{2}\right)$ is the unique (up to isomorphism) quasi-split inner form of $\su(n,k)$.  If $n+k$ is odd, the group $\su\left(\frac{n+k+1}{2},\frac{n+k-1}{2}\right)$ is the unique (up to isomorphism) quasi-split inner form of $\su(n,k)$. Throughout this article, we write $\mathrm{G}=\su(n,1)$ for $n\geq 3$.

\subsection{Root Spaces}
\label{prelims}

We denote the special unitary group of signature $(n,1)$ by:
$$
\mathrm{G}(\Q)=\su(n,1)(\Q)=\{X \in M_{n+1}(\ell) |\ X^\dagger S_{n,1} X = S_{n,1},\ \det(X)=1 \}
$$
over $\Q$. The real Lie group $G=SU (n,1)$ is a semisimple, connected Lie group with finite center. The corresponding real Lie algebra is
$$
\mathfrak{g}_0\coloneqq\mathfrak{su}(n,1)=\{X \in M_{n+1}(\C) |\ X^\dagger S_{n,1}  =-S_{n,1} X,\ \tr(X)=0 \}.
$$
This immediately shows that 
$$
\dim_{\R}(\mathfrak{su}(n,1))= n^2+2n.
$$
As the maximal compact subgroup $K$ we choose
\begin{align*}
    K:&=U_{n+1} \cap G  \\ &=
\{X \in M_{n+1}(\C) |\ X^\dagger S_{n,1} X = S_{n,1},\ X^\dagger X = 1,\ \det(X)=1 \} \cong U_n,
\end{align*}
where 
\begin{align*}
    U_n=\{X \in M_n(\C)|\ X^\dagger X = 1\}.
\end{align*}
The corresponding Lie algebra is
$$
\mathfrak{k}_0=\{X \in M_{n+1}(\C) |\ X^\dagger S_{n,1}  =-S_{n,1} X, X^\dagger = -X,\ \tr(X)=0 \},
$$
with 
$$
\dim(\mathfrak{k}_0)=n^2.
$$
The group $\su (n,1)$ has absolute rank $n$, since $\mathrm{SL}_{n+1}$ has rank $n$, which is the same as the rank of $\mathrm{K}$. This implies that discrete series representations exist, as we will discuss below. A maximal torus $\mathrm{T}$ of $\mathrm{G}$ is given by the subset of diagonal matrices in $\mathrm{G}$, with corresponding Lie algebra $\mathfrak{t}_0$ of the real group $T$.
Note that $T\subseteq K$ and that $T$ contains the one-paramter subgroup specified in \cite[p. 86]{degeorge}.
Our maximal $\R$-torus will be 
$$
A\coloneqq\left\{ \left(\begin{matrix}
a & 0  & b \\
0 & I_{n-1} & 0 \\
b & 0  & a
\end{matrix}\right) | a,b\in \R,\ \ a^2-b^2 = 1 \right\} \cong \R^{\times},
$$
hence $G$ has $\R$-rank $1$.
The corresponding Lie algebra is 
$$
\mathfrak{a}_0=\left\{ \mu(E_{n+1,1}+E_{1,n+1} )|\ \mu \in \R\right\}.
$$
We denote by $\mathfrak{g}, \mathfrak{k}$ and $\mathfrak{t}$ the complexifications of $\mathfrak{g}_0, \mathfrak{k}_0$ and $\mathfrak{t}_0$ respectively. I.e. $\mathfrak{g}= \mathfrak{sl}_{n+1}(\C)$ and $\mathfrak{k}\cong \mathfrak{gl}_n(\C) $. We let $\theta$ be the Cartan involution associated to the maximal compact group $K$, let $\mathfrak{k}\oplus \mathfrak{p}$ be the corresponding Cartan decomposition and $\mathrm{B}(X,Y)$ the Killing form on $\mathfrak{g}$, which is given by
$$
\mathrm{B}(X,Y)=2(n+1)\tr (XY).
$$
We also denote by 
$$
\mathrm{B}_0(X,Y)=\tr (XY)
$$
the corresponding trace form (\cite[§IV.2]{knapp}). We have the Iwasawa decomposition $G=KAN$. Let $P_0=MAN$ be the minimal parabolic subgroup, where $M$ is the centralizer of $A$ in $K$. As in \cite{degeorge}, we normalize the Haar measure on $G$ such that $K$ and $M$ have volume $1$, all discrete subgroups have counting measure and on any subspace $V\subseteq \mathfrak{n}$ we put the measure induced by the Euclidean structure $(X,Y)=-B(X,\theta Y)$ and on $\mathrm{exp}(V)$ the measure induced by $\mathrm{exp}$. On $G$, the measure is given by
\begin{align*}
    \int_G f(g)\, \mathrm{d}g = \frac{1}{\sqrt{2}} \int_K \int_A \int_N f(kan) e^{\rho (\log (a))} \, \mathrm{d}k\mathrm{d}a \mathrm{d}n,
\end{align*}
for $f\in C_c^{\infty}(G)$, where $\rho(H)\coloneqq \frac{1}{2} \tr(ad(H)|_{\mathfrak{n}})$ for $H \in \mathfrak{a}$.

We will first study the root system of $(\mathfrak{g}_0,\mathfrak{a}_0)$ and later that of the complexifications $(\mathfrak{g},\mathfrak{t})$, which is necessary for the parametrization of the discrete series. So let now $\Delta$ be the set of roots of $(\mathfrak{g}_0,\mathfrak{a}_0)$ with a chosen ordering, and let $\Delta^+$ be the positive roots with respect to that ordering. Let 
$$
\mathfrak{g}_0 = (\mathfrak{g}_0)_0 \oplus \sum_{\alpha \in \Delta^+} \mathfrak{g}_\alpha \oplus \sum_{\alpha \in \Delta^+} \mathfrak{g}_{-\alpha}
$$
be the root space decomposition with $(\mathfrak{g}_0)_0=C_{\mathfrak{g}_0}(\mathfrak{a}_0)$. We have a unique simple root, namely $\lambda$, which is given by the map
$$
\mu(E_{n+1,1}+E_{1,n+1} ) \mapsto \mu.$$
$\lambda$ has the root space 
$$
\mathfrak{g}_1\coloneqq\mathfrak{g}_{\lambda} = \left\{ 
\left(\begin{matrix}
0 & x_{12} & \ldots & x_{1n} & 0 \\
-\overline{x_{12}} &  &  &  & \overline{x_{12}} \\
\vdots &  & 0 &  & \vdots  \\
-\overline{x_{1n}} &  &  &  & \overline{x_{1n}} \\
0 & x_{12} & \ldots & x_{1n} & 0 \\
\end{matrix}\right) | \ x_{ij} \in \C
\right\}.
$$
Also 
$$
\mathfrak{g}_2 \coloneqq \mathfrak{g}_{2\lambda} = \left\{ 
\left(\begin{matrix}
-i\mu & 0 & \ldots & 0 & i\mu \\
0 &  &  &  & 0\\
\vdots &  & 0 &  & \vdots \\
0 &  &  &  & 0\\
-i\mu & 0 & \ldots & 0& i\mu \\
\end{matrix}\right) | \ \mu \in \R
\right\}
$$

and 

\begin{align*}
    (\mathfrak{g}_0)_0 = \left\{ 
\left(\begin{matrix}
i\mu_1 & 0 & \ldots & 0 & \mu_2 \\
0 &  & &  & 0 \\
\vdots &  & X &  & \vdots \\
0 &  & &  & 0 \\
\mu_2 & 0 & \ldots & 0 & i\mu_1 \\
\end{matrix}\right) | \ \mu_i \in \R, X \in M_{n-1}(\C)
\right\}.
\end{align*}
 We have $N=\mathrm{exp}(\mathfrak{g}_1 \oplus \mathfrak{g}_2)$ and set $N_1\coloneqq\mathrm{exp}(\mathfrak{g}_1)$ and $N_2\coloneqq\mathrm{exp}(\mathfrak{g}_2)$. We now look at the root systems for the complexifications of the corresponding Lie algebras, as this allows us to parametrize the discrete series representations of $G$. The maximal torus $T$ in $G$ has the real Lie algebra
$$
\mathfrak{t}_0=\left\{\mathrm{diag}(i\lambda_1, \ldots, i\lambda_{n+1}) |\ \lambda_i \in \R,\  \lambda_1 + \ldots \lambda_{n+1}=0 \right\}
$$
with complexification
$$
\mathfrak{t}=\left\{\mathrm{diag}(\lambda_1, \ldots, \lambda_{n+1}) |\ \lambda_i \in \C,\  \lambda_1 + \ldots \lambda_{n+1}=0 \right\}.
$$
We denote by $i\mathfrak{t}_0^*$ the space of all linear forms on $\mathfrak{t}$, which assume imaginary values on $\mathfrak{t}_0$. 
Denote by $\Delta_G$ and $\Delta_K$ the root systems for the pairs $(\mathfrak{g},\mathfrak{t})$ and $(\mathfrak{k},\mathfrak{t})$ respectively. They are subsets of $i\mathfrak{t}_0^*$. We call $\Delta_K$ the compact roots, and $\Delta_n \coloneqq \Delta_G \setminus \Delta_K$ the non-compact roots. Let $\Delta_G^+$ be a set of positive roots for $\Delta_G$ and $\Delta_K^+$ the positive roots of $\Delta_K$, such that $\Delta_K^+ \subset \Delta_G^+$. 
Following \cite{knapp} (IV, §2, page 68), we denote by $\varepsilon_i:\mathfrak{t}\rightarrow \C$ the map
$$
\varepsilon_i(\mathrm{diag}(\lambda_1,...,\lambda_{n+1}))=\lambda_i,
$$
for $1\leq i \leq n+1$. Since $\varepsilon_{n+1}=-\sum_{i=1}^n \varepsilon_i$, we have that $\{\varepsilon_i\}_{i=1}^n$ forms a basis for $\mathfrak{t}^*$ and consequently also for $i\mathfrak{t}_0^*$. We often write elements of this space in coordinates with respect to this basis, i.e. 
\begin{align*}
    (x_1,\ldots,x_n)_{\{\varepsilon_i\}} \ceq \sum_{i=1}^n x_i \varepsilon_i.
\end{align*}
\\
The root spaces are given by
\begin{align*}
    \Delta_G&=\{\varepsilon_i-\varepsilon_j|\ i \neq j, 1\leq i,j\leq n+1\}, \\
\Delta_K&=\{\varepsilon_i-\varepsilon_j|\ i \neq j, 1\leq i,j\leq n\}, \\
\Delta^+_G&=\{\varepsilon_i-\varepsilon_j| 1\leq i < j\leq n+1\}, \\
\Delta^+_K&=\{\varepsilon_i-\varepsilon_j| 1\leq i < j\leq n\}.
\end{align*}
To a root $\varepsilon_i - \varepsilon_j$ we associate an element $H_{\varepsilon_i - \varepsilon_j}$ in $i\mathfrak{t}_0$ such that $B_0(H,H_{\varepsilon_i - \varepsilon_j})=(\varepsilon_i - \varepsilon_j)(H)$ for all $H \in i\mathfrak{t}_0$.
Note that $H_{\varepsilon_i-\varepsilon_j}=E_{ii}-E_{jj}$.
Let $(\cdot,\cdot)$ denote the inner product on $i\mathfrak{t}_0^*$ induced by the trace form $B_0$, i.e. for two roots $\varepsilon_i-\varepsilon_j$ and $\varepsilon_k-\varepsilon_l$ we define $(\varepsilon_i-\varepsilon_j,\varepsilon_k-\varepsilon_l)\coloneqq(\varepsilon_i-\varepsilon_j)(H_{\varepsilon_k-\varepsilon_l})=(\varepsilon_k-\varepsilon_l)(H_{\varepsilon_i-\varepsilon_j})=\delta_{ik}-\delta_{il}-\delta_{jk}+\delta_{jl}$. That means, with respect to the inner product on $i\mathfrak{t}_0^*$, the $\varepsilon_i$ form an
orthonormal system, i.e.
\begin{align*}
    ((\tau_1,\ldots,\tau_n)_{\{\varepsilon_i\}},(\lambda_1,\ldots,\lambda_n)_{\{\varepsilon_i\}})=\tau_1 \lambda_1 + \ldots + \tau_n \lambda_n.
\end{align*}

Let
\begin{align*}
    \delta_G &\coloneqq \frac{1}{2} \sum_{\alpha \in \Delta_G^+} \alpha, \\
    \delta_K &\coloneqq \frac{1}{2} \sum_{\alpha \in \Delta_K^+} \alpha, \\
    \omega_G(\cdot ) &\coloneqq \prod_{\alpha \in \Delta_G^+}   (\cdot, \alpha),  \\
    \omega_K(\cdot ) &\coloneqq \prod_{\alpha \in \Delta_K^+}   (\cdot, \alpha).
\end{align*}
Also, set $\delta_n\coloneqq\delta_G-\delta_K$. Denote by $W_G$ and $W_K$ the corresponding Weyl groups.\\
We have $\omega_K(\delta_K) = \prod_{j=1}^{n} ((j-1)!)$ and
\begin{align*}
\delta_G 
&= \sum_{i=1}^{n} (n+1-i) \varepsilon_i = (n,n-1,n-2,...,1)_{\{\varepsilon_i\}}.
\end{align*}
With these preliminaries about root space, we are in a position to describe the discrete series representations of $G$.
\begin{de}
    An irreducible unitary representation $\pi$ of $G$ is said to be a \textit{discrete series} representation of $G$, if the matrix coefficient $\omega_\pi$ associated to $\pi$ satisfies
\begin{align*}
    \int_G \lVert \omega_\pi\rVert ^2\ dg < \infty. 
\end{align*}
To $\pi$ we may associated a \textit{formal degree} $d_\pi$, defined as the unique positive real number such that 
\begin{align*}
    \int_G \langle \pi(u_1),v_1\rangle \overline{\langle \pi(u_2),v_2\rangle} dg = d_\pi^{-1}\langle u_1,u_2\rangle \langle v_1,v_2 \rangle. 
\end{align*}
We say that $\pi$ is \textit{integrable}, if 
\begin{align*}
    \int_G \lVert \omega_\pi \rVert\ dg < \infty. 
\end{align*}
\end{de}
We define
$$
L_T' \ceq  \{ \tau \in i\mathfrak{t}_0^*|\ (\tau,\alpha) > 0 \ \forall \alpha \in \Delta_G^+,\ \tau+\delta_G\ \text{analytically integral}\}
$$ 
a subset of $i\mathfrak{t}_0^*$ of non-singular elements. For each $\tau  \in L_T'$ we have a discrete series representation $(\pi_\tau,H_\tau)$ as described by Harish-Chandra (\cite[ IX, §7 Theorem 9.20 p.310]{knapp}):


\begin{thm} (Harish-Chandra) \label{harish}\\
Let $\tau \in L'_T $. Then there exists a discrete series representation $\pi_\tau$ of $G$ such that $\pi_\tau \mid_K$ contains with multiplicity one the $K$ type with highest weight 
$$
\Lambda=\tau+ \delta_G -2\delta_K.
$$
$\tau$ is called the Harish-Chandra parameter of that representation.
\end{thm}
With our normalisation of the Haar measure, its formal degree is given by (\cite{degeorge})
\begin{align} \label{formaldegree}
    d_\tau & \ceq d_{\pi_\tau} = \frac{1}{\sqrt{2}}
(2\pi)^{-dim(G/K)/2} \ 2^{-(dim(G/K)-1)/2}\  \frac{|\omega_G(\tau) |}{\omega_K(\delta_K)}  \nonumber \\
&=  \frac{1}{(4\pi)^n} \frac{|\omega_G(\tau) |}{\omega_K(\delta_K)} = \frac{1}{(4\pi)^n} \prod_{i=1}^n  \frac{\left|\prod_{j=i+1}^{n+1} (\tau_i-\tau_j)\right|}{(i-1)!} .
\end{align}
We will now choose Harish-Chandra parameters $\tau$ in such a way, that 
$\pi_\tau$ is integrable. 
In the process, we will concurrently compute the formal degree $d_\tau$, which will be needed later. For a Harish-Chandra parameter $\tau$ we denote
$$
\tau=\rho +\delta_G= (\rho_1,...,\rho_n)_{\{\varepsilon_i\}}+ \delta_G = (\rho_1+n,\rho_2+n-1,...,\rho_n+1).
$$
If $\tau \in L_T'$, we must have that $\tau_i \in \Z$ and that 
\begin{align*}
    \tau_i-\tau_j=(\tau,\varepsilon_i-\varepsilon_j)>0 \quad \ for\ 1\leq i<j \leq n+1.
\end{align*}
In other words, we have that $\tau_1 > \tau_2 > \ldots >\tau_n>\tau_{n+1}$, where $\tau_{n+1} =- \sum_{i=1}^n \tau_i$.
We specify our choices in such a way, that $\pi_\tau$ is integrable:
According to \cite[Theorem p.60]{milicic}, $\pi_\tau$ is integrable if and only if 
    \begin{align} \label{integr}
        |(\tau,\alpha)|>\frac{1}{4}\sum_{\beta\in\Delta_G} |(\alpha,\beta)|\eqqcolon k(\alpha)
    \end{align}
    for all $\alpha \in \Delta_n=\Delta_G \setminus \Delta_K=\{\pm(\varepsilon_k-\varepsilon_{n+1})|\ 1\leq k \leq n\}$. \\
    We see that $k(\varepsilon_k-\varepsilon_{n+1})=n$. So condition \ref{integr} reads as 
    \begin{align} \label{lambdacond}
        |(\tau,\varepsilon_k-\varepsilon_{n+1})|
        =\tau( H_{\varepsilon_k-\varepsilon_{n+1}}) =&\tau_k + \sum_{i=1}^n \tau_i > n \nonumber\\
    \end{align}
    for all $1\leq k \leq n$. Since $|(\tau,-\alpha)|=|(\tau,\alpha)|$ and $k(-\alpha)=k(\alpha)$, this incorporates condition \ref{integr} for the negative roots as well. In summary, the conditions on $\tau$ are:
    \begin{align*}
        \tau_1 &> \tau_2 > \ldots >\tau_n > \tau_{n+1}+n.
        \end{align*}

\subsection{The Selberg Trace Formula} 

We will use the specification of the Selberg trace formula given in \cite{warner} and \cite{degeorge}, to determine the multiplicities $m(\Gamma,\pi_\tau)$ of a integrable discrete series representation $\pi_\tau$ in the space $R_{\Gamma \backslash G} \big|_{L_d^2(\Gamma \backslash G)}$, as outlined in the introduction. Recall that $N_1=\mathrm{exp}(\mathfrak{g}_1)$, $N_2=\mathrm{exp}(\mathfrak{g}_2)$ and $N=\mathrm{exp}(\mathfrak{g}_1 \oplus \mathfrak{g}_2)$. We have the Iwasawa decomposition  $G=KAN$ and the minimal parabolic subgroup $P_0=MAN$, where $M$ is the centralizer of $A$ in $K$. The $G$-conjugates of $P_0$ give the proper $\R$-parabolics of $G$. We note that $N(P_0)/\Gamma \cap N(P_0)$ is compact. Let $\Gamma$ be an arithmetic subgroup of $G$. We call the $\Gamma$-conjugacy classes $P_\Gamma =\{ P^1,...,P^r \} $ of $\Gamma$-cuspidal parabolic subgroups the cusps of $G$, where $P^1=P_0$. 
For $P^i \in P_\Gamma$ let $N^i\coloneqq N(P^i)$. Then $N^i=k_i N k_i^{-1}$ and $P^i=M^iA^iN^i$ with $k_i\in K$. 
Let $\Gamma_1$ be the image of $\Gamma \cap N$ under the natural projection of $N=N_1N_2$ onto the first factor. We write $N_2^i=k_iN_2k_i^{-1}$ and $\Gamma_2^i\coloneqq N_2^i \cap \Gamma$. 

\begin{de}
    Let $G$ be a subgroup of the real group $\mathrm{GL}_n(\R)$. We say that $G$ is \textit{net} if for each $g\in G$ the multiplicative subgroup of $\C^*$ that is generated by the eigenvalues of $g$ is torsion-free.
\end{de}
See Chapter III Section 17 in \cite{borelgroups} for this definition. Ibidem from Proposition 17.4, it follows that any arithmetic subgroup of $G$ has a congruence subgroup that is net.  

\begin{thm} \label{multiplicity} (Theorem 6 in \cite{degeorge}) \\
Let $\Gamma$ be a net arithmetic subgroup of $G=SU(n,1)$. Let $\pi_\tau \in \hat{G}$ be of integrable discrete series with Harish-Chandra parameter $\tau$, where $\tau \in L_T'$. Then 
$$
m(\Gamma,\pi_\tau)=d_\tau \mathrm{vol}(\Gamma  \backslash G) 
$$
if $n$ is odd and
$$
m(\Gamma,\pi_\tau)=d_\tau \mathrm{vol}(\Gamma  \backslash G) +  \kappa \cdot \dim (E_{\tau-\delta_K}) \cdot \sum_{i=1}^r \frac{\mathrm{vol}(\Gamma \cap N^i \backslash N^i)}{\mathrm{vol}(\Gamma_2^i \backslash N_2^i)^{n}}
$$
if $n$ is even. 

\end{thm}
The following terms appear:
\begin{enumerate}
    \item The formal degree $d_\tau$, that we computed above,
    \item The covolume $\mathrm{vol}(\Gamma  \backslash G)$ of $\Gamma$ in $G$, which we will compute in the next Section,
    \item The non-zero constant $\kappa$, defined in \cite[Theorem 6]{degeorge},
    \item The representation $E_{\tau-{\delta_K}}$, which is the irreducible representation of $\mathfrak{k}$ with highest weight $\tau-\delta_K$,
    \item The error terms $$
    \frac{\mathrm{vol}(\Gamma \cap N^i \backslash N^i)}{\mathrm{vol}(\Gamma_2^i \backslash N_2^i)^{n}},
$$ which we will evaluate in Section \ref{error}. 
\end{enumerate}


\section{Prasad's Volume Formula} \label{prasad}
We would like to apply Prasad's Volume Formula (\cite{prasad}, Theorem 3.7) to $\mathrm{G}=\su(n,1)$ with $S=\{\infty\}$ to compute $\mathrm{vol}(\Gamma \backslash G)$. To do this, we will consider principal arithmetic lattices $\Gamma$. We will then look at subgroups $\Gamma(m)$ of $\Gamma$ of finite index $h_m$. We have
$$
\mathrm{vol}(\Gamma(m) \backslash G) = h_m\cdot  \mathrm{vol}(\Gamma \backslash G),
$$
so it suffices to compute $\mathrm{vol}(\Gamma\backslash G)$.

\subsection{The Formula} \label{arithmetic-subgroup}
We interpret $\mathrm{G}_{\Z}$ as the integer points of the $\Q$-algebraic group $\mathrm{G}=\mathrm{SU}(n,1)$ corresponding to the quadratic extension $\ell=\Q(\sqrt{-k})$ of $\Q$. This $\mathrm{G}_{\Z}$ is a lattice in $G$. We call an arithmetic subgroup of $G$ an \textit{arithmetic lattice}, if it is commensurable with $\mathrm{G}_{\Z}$ \cite[§2]{emery}. 
\begin{cor}
    If $n$ is even, the commensurability class of nonuniform arithmetic lattices $\Gamma$ of $G$ is unique. If $n$ is odd, commensurability classes of $G$ are in one-to-one correspondence with $\Q^\times / N_{\ell/\Q}(\ell^\times)$.
\end{cor}
Let $(K_p)_p$ be a coherent collection of parahoric subgroups of $\mathrm{G}(\Q_p)$. As in \cite[§3.4]{prasad} or \cite[§2.2 ]{emery}, an arithmetic subgroup $\Gamma$ of $G$ can be constructed: 
By strong approximation, 
$$
G \cdot \prod_{p\ prime} K_p \cdot \mathrm{G}(\Q)=\mathrm{G}(\mathbb{A}).
$$
\begin{de} \label{defgamma}
    We say that an arithmetic subgroup $\Gamma$ of $G$ is a \textit{principal arithmetic lattice} if it is the image of 
    $$
    \mathrm{G}(\Q)\cap \left(G\cdot \prod_{p\ prime} K_p \right)
    $$
    under the natural projection
    $$
     G \cdot \prod_{p\ prime} K_p \rightarrow G,
    $$
    where $\mathrm{G}(\Q)$ is understood as embedded diagonally into $\mathrm{G}(\mathbb{A})$ and $(K_p)_p$ is a coherent collection of parahoric subgroups of $\mathrm{G}(\Q_p)$.
\end{de}
We denote by $T_\ell$ the set of finite places $p$ where the the local group $\mathrm{G}(\Q_p)$ is not quasi-split.
Let $\Gamma$ be a principal arithmetic lattice of $G$ and let $(K_p)_p$
be the coherent collection of parahoric subgroups $K_p$ of $\mathrm{G}(\Q_p)$ corresponding to $\Gamma$.
We denote the quasi split unitary group, such that $\mathrm{G}$ is an inner form of it by $\mathscr{G}$. We noted above, that $\mathscr{G}=\mathrm{SU}(\frac{n+1}{2},\frac{n+1}{2})$ if $n$ is odd and $\mathscr{G}=\mathrm{SU}(\frac{n}{2}+1,\frac{n}{2})$ if $n$ is even. As described in \cite[2.2]{prasad}, we let $\mathrm{G}_p$ be the smooth affine $\Z_p$-group scheme associated with the parahoric subgroup $K_p$ of $\mathrm{G}(\Q_p).$ Analogously, we let $\mathscr{G}_p$ be the canonical smooth affine $\Z_p$-group scheme of $\mathscr{G}(\Q_p)$ associated by Bruhat-Tits theory \cite[Section 4]{gr} and let $\mathscr{K}_p$ be the corresponding parahoric subgroup.
Furthermore, we denote by $\overline{\mathrm{G}}_p$ and $\overline{\mathscr{G}}_p$ the the reduction modulo $p$ of $\mathrm{G}_p$ and $\mathscr{G}_p$ over the residue field $\mathbb{F}_p$. Both are connected and admit a Levi decomposition
$$
\overline{\mathrm{G}_p} = \overline{\mathrm{M}_p} \cdot \mathrm{R}_u(\overline{\mathrm{G}_p})
$$
and 
$$
\overline{\mathscr{G}_p} = \overline{\mathscr{M}_p} \cdot \mathrm{R}_u(\overline{\mathscr{G}_p}),
$$
where $\mathrm{R}_u$ denotes the unipotent radical. Computing the corresponding integral models, we can determine the groups $\overline{M}_p$ and $\overline{\mathscr{M}}_p$. This is necessary for Prasad's Volume formula. 
Namely, we apply \cite[Theorem 3.7]{prasad} and \cite[Proposition 3.6]{emery} (note that we omit the Tamagawa number in the formula, since it is $1$ in our case anyway):
\begin{align} \label{pra}
    \mu_{EP}(\Gamma\backslash G)=\chi(\C \mathbb{P}^n) D_\ell^{s(n)} \left( \prod_{r=1}^n \frac{r!}{(2\pi)^{r+1}}\right) \cdot \prod_p \frac{p^{\frac{1}{2}(\dim(\overline{M}_p) + \dim(\overline{\mathscr{M}}_p))}}{|\overline{M}_p(\mathbb{F}_{p})|},
\end{align}
where 
\begin{align*}
    s(n)\coloneqq \left \{ 
    \begin{matrix}
        \frac{(n-1)(n+2)}{4} & \text{ if $n$ is odd}\\
        \frac{n(n+3)}{4} & \text{ if $n$ is even}
    \end{matrix}
    \right.
\end{align*}
and $\chi(\C \mathbb{P}^n)$ is the Euler-Characteristic of the compact dual $U/K=\C \mathbb{P}^n$ of $G$.
Since $\overline{\mathscr{M}_p} \cong \overline{{M}_p} $ for $p \not \in T_\ell$, we can rewrite the product in (\ref{pra}) as in \cite[Section 3.1]{emery} with:
\begin{align}
    \prod_p \frac{p^{\frac{1}{2}(\dim(\overline{M}_p) + \dim(\overline{\mathscr{M}}_p))}}{|\overline{M}_p(\mathbb{F}_{p})|} = 
    \prod_p \frac{p^{\dim(\overline{\mathscr{M}}_p)}}{|\overline{\mathscr{M}}_p(\mathbb{F}_{p})|} \cdot \prod_{p\in T_\ell} \lambda_p,
\end{align}
with the rational factors
\begin{align} \label{lambda}
    \lambda_p \ceq p^{\frac{1}{2}(\dim(\overline{M}_p) - \dim(\overline{\mathscr{M}}_p))} \frac{|\overline{\mathscr{M}}_p(\mathbb{F}_{p})|}{|\overline{M}_p(\mathbb{F}_{p})|}.
\end{align}
We have 
$$
\overline{\mathscr{M}}_p= \left \{
\begin{matrix}
\mathrm{SL}_{n+1}  & \text{ $p$ split}\\
\su_{n+1}  & \text{ $p$ inert}\\
\mathrm{Sp}_{n+1}   & \text{ $p$ ramified}
\end{matrix}\right.
$$
if $n$ is odd and 
$$
\overline{\mathscr{M}}_p= \left \{
\begin{matrix}
\mathrm{SL}_{n+1}   & \text{ $p$ split}\\
\su_{n+1}   & \text{ $p$ inert}\\
\mathrm{SO}_{n+1}   & \text{ $p$ ramified}
\end{matrix}\right.
$$
if $n$ is even.
Thus (\cite[Remark 1.2.3]{finiteclassical} and \cite[2.6]{finiteclassical2}):
$$
\frac{p^{\dim(\overline{\mathscr{M}}_p)}}{|\overline{\mathscr{M}}_p(\mathbb{F}_{p})|} =  \left \{
\begin{matrix}
\prod_{i=2}^{n+1} \frac{1}{1-\frac{1}{p^i}} & \text{ $p$ split}\\
\prod_{i=2}^{n+1} \frac{1}{1-\frac{(-1)^i}{p^i}} & \text{ $p$ inert} \\
\prod_{i=1}^{\frac{n+1}{2}} \frac{1}{1-\frac{1}{p^{2i}}}  & \text{ $p$ ramified}.
\end{matrix}\right.
$$
This shows that 
\begin{align} \label{zetaproductodd}
   \prod_{p} \frac{p^{\dim(\overline{\mathscr{M}}_p)}}{|\overline{\mathscr{M}}_p(\mathbb{F}_{p})|}  =
   \zeta(2)L_\ell(3)\cdot ...\cdot \zeta(n+1),
\end{align}
if $n$ is odd and
\begin{align} \label{zetaproducteven}
   \prod_{p} \frac{p^{\dim(\overline{\mathscr{M}}_p)}}{|\overline{\mathscr{M}}_p(\mathbb{F}_{p})|}  =
   \zeta(2)L_\ell(3)\cdot ...\cdot L_\ell(n+1),
\end{align}
if $n$ is even.

To determine $\vol(\Gamma \backslash G)$ with the measure chosen above, we proceed as in \cite[proof of Proposition 4.1]{Harald} and use the fact that
\begin{align}
\mathrm{vol}_{\mu} &= \mathrm{vol}_{g_0}(\C \mathbb{P}^n) \frac{\mathrm{vol}_{EP}}{\chi(\C \mathbb{P}^n)} \\
 &= \mathrm{vol}_{g_0}(\C \mathbb{P}^n) \cdot |D_\ell|^{s(n)} \left( \prod_{r=1}^n \frac{r!}{(2\pi)^{r+1}}\right) \cdot \prod_{p} \frac{p^{\dim(\overline{\mathscr{M}}_p)}}{|\overline{\mathscr{M}}_p(\mathbb{F}_{p})|} \cdot \prod_{p \in T_\ell} \lambda_p,
\end{align}
where $g_0$ denotes the canonical Riemannian metric on $\C \mathbb{P}^n$. By \cite[Equation 3.7]{vol}, we have

$$
\mathrm{vol}_{g_0}(\C \mathbb{P}^n)=\frac{\pi^n}{n!}.
$$ 
We combine these results to
\begin{thm} \label{volume-final}
$$
\vol (\Gamma \backslash G)= \pi^n |D_\ell|^{\frac{(n-1)(n+2)}{4}} \left( \prod_{i=1}^n \frac{(i-1)!}{(2\pi)^{i+1}}\right) \cdot \prod_{k=2}^{n+1} Z_{\ell}(k)  \cdot \prod_{p \in T_\ell}\lambda_p
$$
for $n$ odd and
$$
\vol (\Gamma \backslash G)= \pi^n |D_\ell|^{\frac{n(n+3)}{4}} \left( \prod_{i=1}^n \frac{(i-1)!}{(2\pi)^{i+1}}\right) \cdot \prod_{k=2}^{n+1} Z_{\ell}(k)
$$
for $n$ even.
\end{thm}

\subsection{Parahoric Subgroups of Maximal Volume}
As an example, we consider the case where $(K_p)_p$ is a coherent collection of parahoric subgroups of maximal volume in $\mathrm{G}(\Q_p)$.
This ensures that $K_p$ is special for all $p$ and if $p\not \in T_\ell$ is not ramified in $\ell$, then $K_p$ is hyperspecial (\cite[Section 3.1]{emery}). The covolume of the principal arithmetic lattice $\Gamma$ is now the same for parahoric subgroups of maximal volume, i.e. independent of the explicit choice of the individual $K_p$. To determine the factors $\lambda_p$, we study special unitary groups over local fields. We will then define the parahoric subgroups, the integral models and compute the reductive quotient over the special fibres in relevant case. 
Let $L/\Q_p$ be an imaginary quadratic field extension. Now for $q\in \N$ we set $V\coloneqq L^q$ and we look at hermitian forms $h$ on $V$ given by (with a chosen basis)
$$
h(\mathbf{x},\mathbf{y})\coloneqq \mathbf{y}^\dagger S_h \mathbf{x}
$$
where $\overline{\cdot}$ denotes the non-trivial element in the Galois group of $L/\Q$ and $S_h$ is a hermitian matrix with respect to $\overline{\cdot}$. We define the following (\cite[Section 3]{hanke}):
\begin{enumerate}
    \item $H\coloneqq L^2$ is the split rank-two hermitian space over $L$ with hermitian form $h$ given by 
    $$
    S_h=\left(\begin{matrix}
0 & 1  \\
1 & 0  \\
\end{matrix}\right).
    $$
    \item Let $A_{L}$ be the ring of integers of $L$. 
    \item We denote by $\mathscr{D}$ the different of $L/\Q_p$.
    \item Let $\Delta\coloneqq A_{L} \oplus \mathscr{D}^{-1} $ be the maximal $A_{L}$-lattice in $H$.
    \item Let $N_{L / \Q_p}$ be the norm of the extension $L$.
\end{enumerate}

Now \cite{hanke} describes the two isomorphism classes of such hermitian spaces over local fields. These are classified by their discriminant, which is an element of $\Q_p^\times /N_{L / \Q_p}(L^\times)$. For each class of $\Q_p^\times /N_{L / \Q_p}(L^\times)$, we pick a representative $\alpha \in \Z_p$ of minimal valuation, thus $0\leq v_p(\alpha) \leq 1$: 
\begin{enumerate}
    \item In the odd case, $q=2m+1$, the two isomorphism classes look as follows: \\
For $\alpha \in \Q_p^\times$ let $V_\alpha\coloneqq H^m \oplus E_{\alpha}$, where $E_{\alpha}$ is the one dimensional hermitian space with hermitian form $h(x,y)=\alpha x \overline{y}$.\\
Now $\mathrm{disc}(h)=(-1)^m \alpha$. 
We define the corresponding maximal lattice 
$$
\Lambda_{\alpha}\coloneqq \Delta^m \oplus A_{L} 
$$

For both classes, the corresponding unitary groups are isomorphic, since we can switch between the classes by multiplying $S_h$ with an element of $\Q_p^\times$, leaving the unitary group invariant. Moreover, both groups are quasi-split.  
\item In the even case, $q=2m$, the two isomorphism classes look as follows: \\
For $\alpha \in N_{L / \Q_p}(L^\times)$ let $V_\alpha\coloneqq H^m$ and the maximal lattice
$$
\Lambda_{\alpha}=\Delta^m.
$$

Now $\mathrm{disc}(h)=(-1)^m= (-1)^m \alpha \in \Q_p^\times /N_{L / \Q_p}(L^\times)$. \\
For $\alpha \not\in N_{L / \Q_p}(L^\times)$ let $V_\alpha\coloneqq H^{m-1} \oplus D$, where $D$ is the division algebra $L \oplus L\cdot z$ over $\Q_p$ with the multiplication
\begin{align} 
    z^2 &= \alpha,    \nonumber \\ 
    x \cdot z &= z \cdot \overline{x}.
\end{align}
Let $A_D$ be its ring of integers and take the hermitian form 
\begin{align} \label{hermitian}
    h(x_1+x_2 \cdot z, y_1 + y_2 \cdot z)=x_1 \overline{y_1} - \alpha x_2 \overline{y_2}.
\end{align}

In this case, $\Lambda_{\alpha}=\Delta^{m-1} \oplus A_D$ and 
$\mathrm{disc}(h)=(-1)^{m} \alpha$. 
Here, only the unitary group attached to $V_\alpha$ for $\alpha \in N_{L / \Q_p}(L^\times)$ is quasi-split. 

\end{enumerate}
With this classification of hermitian spaces over local fields, we make the following observations: If $n$ is even, every $\mathrm{G}(\Q_p)$ is quasi-split and hence $T_\ell=\emptyset$. For $n$ odd, we distinguish two cases: If $p$ is unramified, every unit is a norm, which implies that our global hermitian form over $\ell$ (with discriminant $-1$) descends to the local hermitian space corresponding to the identity in $N_{L / \Q_p}(L^\times)$, which is quasi-split (hence $p\not\in T_\ell$). If $p$ is ramified, denote by $\mathfrak{p}$ the unique prime ideal that lies above $p$. Then $p\in T_\ell$ if and only if $-1=\mathrm{disc}(h_{\mathrm{G}(\Q_p)})\neq \mathrm{disc}(h_{\mathscr{G}(\Q_p)})=(-1)^{\frac{n+1}{2}}$ in $\Q_p^\times / N_{\ell_{\mathfrak{p}} / \Q_p}(\ell_{\mathfrak{p}}^\times)$, so that is the places where $(-1)^{\frac{n-1}{2}} \not\in N_{\ell_{\mathfrak{p}} / \Q_p}(\ell_{\mathfrak{p}}^\times)$. For these primes $p$, we have (\cite[Proposition 3.9]{hanke}),  
                        \begin{align*}
                            \overline{M}(\mathbb{F}_p) &= \{\mathrm{diag}(X,Y)|X\in \mathrm{Sp}_{n-1}(\mathbb{F}_p), Y\in  \mathrm{O}^-_2(\mathbb{F}_p), \det (Y)=1 \} \\
                        &\cong \mathrm{Sp}_{2m-2}(\mathbb{F}_p) \times \mathrm{SO}^-_2(\mathbb{F}_p), 
                        \end{align*}
    if $p$ is odd, where $\mathrm{SO}^-_2$ is defined in \cite{finiteclassical2} 2.6. For $p=2$, the matrices in $\mathrm{O}^-_2(\mathbb{F}_p)$ already have determinant $1$ and so 
    \begin{align*}
                            \overline{M}(\mathbb{F}_2) 
                        &\cong \mathrm{Sp}_{2m-2}(\mathbb{F}_2) \times \mathrm{O}^-_2(\mathbb{F}_2).
                        \end{align*}
         Hence for $p\in T_\ell$ we have 
\begin{align*}
    \lambda_p &=p^{\frac{1}{2}(\dim(\mathrm{Sp}_{n-1}\times \mathrm{SO}^-_2 - \dim(\mathrm{Sp}_{2m}} \frac{|\mathrm{Sp}_{n+1}(\mathbb{F}_{p})|}{|\mathrm{Sp}_{n-1}\times \mathrm{SO}^-_2(\mathbb{F}_{p})|}= \\
    &= \frac{p^{n+1}-1}{p+1},
\end{align*}
due to \cite[Remark 1.2.3]{finiteclassical} and \cite[2.6]{finiteclassical2}.

\section{The Error Term} \label{error}

Let $\Gamma$ be a principal arithmetic lattice. Recall that we intend to look at normal subgroups $\Gamma(m)$ of $\Gamma$ of finite index. We have to define what we mean by the principal congruence subgroups $\Gamma(m)$. We defined the $\Q$-algebraic group $\mathrm{G}=\mathrm{SU}(n,1)$ as the unitary group of the hermitian form with signature $(n,1)$ over the imaginary field extension $\ell=\Q(\sqrt{-k})$.
\begin{align*}
    \mathrm{G}(\Q)= \left\{ X\in M_{n+1}(\ell)|\ X^\dagger S_{n,1} X= S_{n,1}, \det(x)=1 \right\}.
\end{align*}
By restriction of scalars of the field $\ell$, we may embed $\varphi: \mathrm{G}(\Q) \hookrightarrow \mathrm{GL}_{2(n+1)}(\Q)$. Set 
$$
\mathrm{G}_{\Z}\coloneqq\varphi^{-1} (\mathrm{GL}_{2(n+1)}(\Z)).
$$
This depends on the embedding $\varphi$, but the commensurability class of $\mathrm{G}_{\Z}$ does not. Let $m$ be coprime to $2k$ and define 
$$
\pi_m:\mathrm{GL}_{2(n+1)}(\Z) \twoheadrightarrow \mathrm{GL}_{2(n+1)}(\Z / m \Z )
$$
the reduction modulo $m$ map. 
We computed the co-volume of a principal arithmetic lattice $\Gamma$ in the real group $G$ and this $\Gamma$ is commensurable with $\mathrm{G}_{\Z}$ \cite{emery}.
\begin{de} \label{congruence-subgroup}
We denote with
$$
\Gamma(1)\coloneqq\Gamma \cap \mathrm{G}_{\Z},
$$
which has finite index in both $\Gamma$ and $G(\Z)$. 
Let
$$
\Gamma(m)\coloneqq\Gamma(1) \cap \ker ((\pi_m\circ \varphi)|_{\mathrm{G}_{\Z}}),
$$
for $m\geq 3$.
\end{de}
This implies that $\Gamma(m)$ is normal in $\Gamma(1)$. Furthermore, since $m\geq 3$, the group $\Gamma(m)$ is torsion free \cite[Lemma 1]{serre}. There are infinitely many $m$ such that $\Gamma(m)$ is net (\cite[III, Proposition 17.4]{borelgroups}).
Now we know that 
$$
\mathrm{vol}(\Gamma(m) \backslash G) = h_m\cdot  \mathrm{vol}(\Gamma \backslash G),
$$
where $h_m$ is the index of $\Gamma(m)$ in $\Gamma$.

\subsection{Notation}

In the last chapter, we got a closed formula for $m(\Gamma(m),\pi)$ for the case that $n$ is odd. For $n$ even, we have to evaluate the error term 

\begin{align} \label{errror}
  \kappa (-1)^{n(\tau)} \dim (E_{\tau-\delta_K}) \cdot \sum_{i=1}^s \frac{\mathrm{vol}(\Gamma(m) \cap M^i \backslash M^i)}{\mathrm{vol}(\Gamma(m)\cap M_2^i \backslash M_2^i)^{n}}
\end{align}
in \ref{multiplicity}, where $M^i=\mathrm{R}_u(Q^i)$ are the unipotent radicals of the cusps $Q^1,...,Q^s$ of $\Gamma(m)$.
The proof of Theorem 10 in \cite{degeorge} gives the following:
$$
\sum_{i=1}^s \frac{\mathrm{vol}(\Gamma(m) \cap M^i \backslash M^i)}{\mathrm{vol}(\Gamma(m)\cap M_2^i \backslash M_2^i)^{n}}=h_N \sum_{j=1}^r \frac{\mathrm{vol}(\Gamma(1) \cap N^j \backslash N^j)}{\mathrm{vol}(\Gamma(m) \cap N_2^j \backslash N_2^j)^{n}},
$$
where $N^j=\mathrm{R}_u(
P^j)$ for the cusps $P^1,...,P^r$ of $\Gamma(1)$. So we have to determine $\mathrm{vol}(\Gamma(1) \cap N^j \backslash N^j)$ and $\mathrm{vol}(\Gamma(m) \cap N_2^j \backslash N_2^j)$ or at least give an upper bound for the quotient. 
 We introduce the notation
 $$
 A(a,b)\coloneqq\left(\begin{matrix}
a & 0 & ... & b \\
0 & 1 & ... & 0 \\
... & & & \\
b & 0 & ... & a
 \end{matrix}\right),
 $$
 so $A=\{A(a,b)|\ a,b\in\R\}$. \\
 Now 
 $$
 N_2=\mathrm{exp}(\mathfrak{g}_2)=\{u_2(\mu)|\ \mu\in\R\}, 
 $$
 where 
 $$
 u_2(\mu)\coloneqq\left(\begin{matrix}
1-i\mu & 0 & ... & i\mu \\
0 & 1 & ... & 0 \\
... & & & \\
-i\mu & 0 & ... & 1+i\mu
 \end{matrix}\right)
 $$
 and 
 $$
 N=\mathrm{exp}(\mathfrak{g}_1 \oplus \mathfrak{g}_2)=\{u(\mathbf{x},\mu)|\ \mu\in\R,\ \mathbf{x}=(x_1,...,x_{n-1})\in \C^{n-1} \}, 
 $$
 where 
 $$
u(\mathbf{x},\mu)\coloneqq u_2(\mu) + \left(\begin{matrix}
-\frac{|\mathbf{x}|^2}{2} & x_1 & x_2 & ... & \frac{|\mathbf{x}|^2}{2} \\
-\overline{x_1} & 0 & 0 & ... & \overline{x_1} \\
...& & & & \\
-\frac{|\mathbf{x}|^2}{2} & x_1 & x_2 & ... & \frac{|\mathbf{x}|^2}{2}
 \end{matrix}\right).
 $$
 Above, we have set
 $$
 |\mathbf{x}|^2\coloneqq\langle \mathbf{x},\mathbf{x}\rangle,
 $$
 where
 $$
 \langle \mathbf{x}, \mathbf{y} \rangle \coloneqq \sum_{i=1}^{n-1} x_i\overline{y_i}. 
 $$
 Note that 
 $$
 u_2(\mu_1)\cdot u_2(\mu_2)=u_2(\mu_1+\mu_2)
 $$
 and 
 $$
 u(\mathbf{x}^1,\mu_1)\cdot  u(\mathbf{x}^2,\mu_2) = 
 u(\mathbf{x}^1+\mathbf{x}^2, \mu_1+\mu_2+\mathrm{Im}(\langle \mathbf{x}^1,\mathbf{x}^2\rangle )).
 $$
\\
In Section \ref{prelims}, we defined $M$ as the centralizer of $A$ in $K$. This means, that 
\begin{align*}
    M &=\{X\in K|\ A(a,b)X=XA(a,b)\ \forall a,b\in \R\}\\
    &= \{\mathrm{diag}(e^{i\theta},Y,e^{i\theta})|\ Y^\dagger Y= 1,\ \det(Y)=e^{-2i\theta}  \}.
\end{align*}
We have $P_0=MAN$. Recall that we sum the Error terms 
$$
\frac{\mathrm{vol}(\Gamma(1) \cap N^j \backslash N^j)}{\mathrm{vol}(\Gamma(m) \cap N_2^j \backslash N_2^j)^{n}},
$$
where $j$ runs through an indexation $P^1,...,P^r$ of the cusps of $G$ (that is, $\Gamma(1)$-conjugacy classes of $\Gamma(1)$-cuspidal parabolic subgroups of $G$). For an arbitrary $P^i\in P_\Gamma$, we have that 
\begin{align*}
    P^i&=M^iA^iN^i \\
    N^i&=k_i^{-1} N k_i \\
    N_2^i &= k_i^{-1} N_2 k_i, 
\end{align*}
where $k_i\in K$. 
\subsection{Estimating the Error Term} \label{errorterm}
Assume we work with $n$ even, the only case where the error terms appear. Here, we know that there is only one commensurability class of nonuniform arithmetic lattices.
We want to estimate the size of $\mathrm{vol}(\Gamma(m) \cap N_2^j \backslash N_2^j)$.
    We know that $N_2^j=(k^j)^\dagger N_2 k^j$, for some $k^j\in K$. Now let 
    $$
    k^j= \left(\begin{matrix}
        k_{11} & k_{12} & \ldots & 0 \\
        k_{21} & k_{22} & \ddots & 0 \\
        \vdots & \vdots & \ddots & 0 \\
        0 & 0& \ldots & k_{n+1,n+1}
    \end{matrix}\right).
    $$
    Then an element of $\Gamma(m) \cap N_2^j$ is of the form $u_2^j(\mu)\coloneqq(k^j)^\dagger u_2(\mu)k^j$ and its $(n+1,n+1)$-entry is still $1 + i \mu$. But this implies that $\mu = \sqrt{k} zm$ for $z \in \Z$. We thus see that 
    $\Gamma(m) \cap N_2^j \subseteq \{u_2^j(\sqrt{k} mz)| z\in \Z\}\subseteq N_2^j$ and so
\begin{align} \label{volume-N2}
    \mathrm{vol}(\Gamma(m) \cap N_2^j \backslash N_2^j) &\geq \mathrm{vol}(\{u_2^j(\sqrt{k} mz)| z\in \Z\} \backslash N_2^j)\nonumber\\
    &=\mathrm{vol}(\{u_2(\sqrt{k} mz)| z\in \Z\} \backslash N_2) 
    = \sqrt{\frac{k}{8(n+1)}} m,
\end{align}
with the normalization of the Haar measure chosen in Section \ref{prelims}.

We plug this into the error term (\ref{errror}) and get:

\begin{align*}
 & \left| \kappa (-1)^{n(\tau)} \dim (E_{\tau-\delta_K}) \cdot \sum_{i=1}^s \frac{\mathrm{vol}(\Gamma(m) \cap M^i \backslash M^i)}{\mathrm{vol}(\Gamma(m)\cap M_2^i \backslash M_2^i)^{n}} \right|\\
\leq & |\kappa| \dim (E_{\tau-\delta_K}) \cdot h_m \sum_{j=1}^r \frac{\mathrm{vol}(\Gamma(1) \cap N^j \backslash N^j)}{\mathrm{vol}(\Gamma(m)\cap N_2^j \backslash N_2^j)^{n}} \\
\leq  & |\kappa|  \dim (E_{\tau-\delta_K}) \cdot \frac{h_m}{\left(\sqrt{\frac{k}{8(n+1)}} m \right)^n}\sum_{j=1}^r \mathrm{vol}(\Gamma(1) \cap N^j \backslash N^j) \\
\eqqcolon  & C \frac{h_m}{m^n}.
\end{align*}
When we note that 
$$
\mathrm{vol}(\Gamma(m) \backslash G) = h_{m}\cdot \mathrm{vol}(\Gamma \backslash G),
$$
we obtain our main result:

\begin{thm} \label{mult}
Let $\Gamma$ be a principal arithmetic lattice of $G=SU(n,1)$ and let $\Gamma(m)$ be a net principal congruence subgroup (as in Definition \ref{congruence-subgroup}) of index $h_m$ in $\Gamma$, with $n\geq 3$ and $m\geq 3$. Let $\pi_\tau \in \hat{G}$ be of integrable discrete series with Harish-Chandra parameter $\tau=(\tau_1,.\ldots,\tau_n)_{\{\varepsilon_i\}}$ with $\tau_{n+1}\ceq - \sum_{i=1}^n \tau_i$, where
\begin{align*}
        \tau_1 &> \tau_2 > \ldots >\tau_n > \tau_{n+1}+n.
        \end{align*}
Then we have
\begin{align*}
    & m(\Gamma(m),\pi_\tau) =d_\tau \mathrm{vol}(\Gamma  \backslash G)h_m \\
    &= \frac{h_m}{4^n} \prod_{i=1}^n  \frac{\left|\prod_{j=i+1}^{n+1} (\tau_i-\tau_j)\right|}{(2\pi)^{i+1}} 
    \cdot |D_\ell|^{\frac{(n-1)(n+2)}{4}}  \cdot \prod_{k=2}^{n+1} Z_{\ell}(k)  
    \cdot \prod_{p \in T_\ell}\lambda_p 
\end{align*}
if $n$ is odd and 
\begin{align*}
& m(\Gamma(m),\pi_\tau) =d_\tau \mathrm{vol}(\Gamma  \backslash G)h_m + O\left( \frac{h_m}{m^n}\right) \\
&= h_m \left( \frac{1}{4^n} \prod_{i=1}^n  \frac{\left|\prod_{j=i+1}^{n+1} (\tau_i-\tau_j)\right|}{(2\pi)^{i+1}} \cdot |D_\ell|^{\frac{n(n+3)}{4}}  \cdot \prod_{k=2}^{n+1} Z_{\ell}(k)  + O\left( \frac{1}{m^n}\right)\right)
\end{align*}
if $n$ is even.
The finite set $T_\ell$ consists of those primes $p$ where $\mathrm{G}(\Q_p)$ is not quasi-split. The rational coefficients $\lambda_p$ are defined in \eqref{lambda}.
\end{thm}

\section{Application I: Cuspidal Cohomology} \label{cuspcoh}

As before, we let $\tau=\rho + \delta_G$. In Theorem 5.3 of \cite{zuck}, the $\mathfrak{g}$-module $A_{\mathfrak{q}}(\rho)$ is identified as the discrete series representation of Harish-Chandra parameter $\tau$, if we let $\mathfrak{q}$ be the standard parabolic sub-algebra of $\mathfrak{g}$ corresponding to the parabolic subgroup $P$ of $G$, given by the upper triangular matrices in $G$. We require 
    \begin{align*}
        (\rho,\alpha)\geq 0\ \forall \alpha \in \Delta_G^+,
    \end{align*}
    which is equivalent to 
    \begin{align} \label{cohom-condition}
        \tau_i +i &\geq \tau_j + j \quad \ &\text{for   $1\leq i <j\leq n$}, \nonumber \\
        \tau_i  &\geq  \tau_{n+1} + \frac{n(n+1)}{2} + n+1 -i \quad \ &\text{for $1\leq i \leq n$},
    \end{align}
    where again we have set $\tau_{n+1} = -\sum_{i=1}^n \tau_i $.
    According to Theorem 5.3 and Theorem 5.5 in \cite{zuck} we have that $\pi_\tau$ is cohomological with respect to $F_{\rho}$,
    the finite dimensional irreducible representation of $\mathfrak{g}$ of lowest weight $-\rho$. \\ 
    More precisely, we have that (note that $\dim \mathfrak{u} \cap \mathfrak{p}=n$)
    \begin{align*}
        H^i(\mathfrak{g},\mathfrak{k},\pi_\tau \otimes F_{\rho}) \cong H^{i-n}(\mathfrak{t},\mathfrak{t}\cap \mathfrak{k},\C) \cong \mathrm{Hom}_{\mathfrak{t}}(\Lambda^{i-n}(0),\C),
    \end{align*}
    which means that 
    \begin{align} \label{cohomology}
        H^i(\mathfrak{g},\mathfrak{k},\pi_\tau \otimes F_{\rho}) \cong \left\{ 
        \begin{matrix}
            \C & \text{if $i=n$}\\
            0 & \text{else}.
        \end{matrix}
        \right.
    \end{align}
    Also, for $F \not \cong F_{\rho}$, 
    \begin{align*}
        H^i(\mathfrak{g},\mathfrak{k},\pi_\tau \otimes F)= 0. 
    \end{align*}
    
 We have defined the arithmetic subgroup $\Gamma(1)$ of $G$ above. Consider a tower of arithmetic subgroups $\{\Gamma(m_i)\}_{i=1}^\infty$ of $G$, where $m_1=1$ and $\Gamma(m_i)$ is a normal congruence subgroup of finite index in $\Gamma(1)$, as in the previous section. We also assume 
$$
\bigcap_{i=1}^\infty \Gamma(m_i)=\{1\}. 
$$
With the notation from above, we get the following: 
\begin{thm} (Theorem 10 in \cite{degeorge}) \\
For $\pi_\tau$ of integrable discrete series, we have
$$
\lim_{i \rightarrow \infty} \frac{m(\Gamma(m_i),\pi_\tau)}{\mathrm{vol}(\Gamma(m_i) \backslash G)} = d_\tau.
$$

\end{thm}

With the estimates derived in the last chapter, we can give an explicit bound for this growth rate. The equation for the even case in Theorem \ref{mult} shows that
$$
\frac{m(\Gamma(m_i),\pi) }{\mathrm{vol}(\Gamma(m_i) \backslash G)}\geq d_\tau - \frac{C}{ \mathrm{vol}(\Gamma \backslash G) m_i^n  },
$$
because 
$$
\mathrm{vol}(\Gamma(m_i) \backslash G) = h_{m_i}\cdot \mathrm{vol}(\Gamma \backslash G).
$$
In particular, this tells us that the multiplicity is strictly positive (that, of course, means non-zero) as soon as 
$$
m_i \geq \left|\frac{C}{d_\tau \mathrm{vol}(\Gamma \backslash G)} \right|^{\frac{1}{n}}.
$$
\color{black}
We want to apply this result to get some growth estimates of cuspidal cohomology.
Let $X\coloneqq G/K$ be the Riemannian symmetric space associated to $G$. For $m_i \geq 3$ we have that $\Gamma(m_i)$ is a torsion free arithmetic subgroup of $G$. Then 
$$
S(\Gamma(m_i))\coloneqq\Gamma(m_i) \backslash X
$$
is a locally symmetric space and a smooth, non-compact manifold of finite volume. Let $F_{\rho}$ be the irreducible, finite dimensional representation of $G$ of lowest weight $-\rho$.
We denote by $\Tilde{F_{\rho}}$ the locally constant sheaf corresponding to $F_{\rho}$. 
Recalling the definition of cuspidal cohomology (\cite[1.2]{Harald}), we have 
$$
H_{cusp}^* (S(\Gamma(m_i)),\Tilde{F_{\rho}})\coloneqq H^*(\mathfrak{g},K,L_{cusp}^2(\Gamma(m_i) \backslash G)^\infty \otimes F_{\rho}).
$$
Now by \cite[Proposition 5.6]{laplacian} we have
$$
H_{cusp}^* (S(\Gamma(m_i)),\Tilde{F_{\rho}})\coloneqq \bigoplus _{\pi \in \hat{G}_{coh}} H^*(\mathfrak{g},K,\pi_{(K)} \otimes F_{\rho})^{m(\Gamma(m_i),\pi)}.
$$
Using Theorem \ref{mult} and Equation \ref{cohomology} we get the following:
\begin{thm} 
Let $\{\Gamma(m_i)\}_{i=1}^\infty$ be a sequence of arithmetic subgroup of $G=SU(n,1)$, as above. Let $\pi_\tau \in \hat{G}$ be of integrable discrete series with Harish-Chandra parameter $\tau=\rho + \delta_G$, where $\tau \in L_T'$ satisfies
 \begin{align*} 
        \tau_i +i &\geq \tau_j + j \quad \ &\text{for $1\leq i <j\leq n$}, \nonumber \\
        \tau_i  &\geq  \tau_{n+1} + \frac{n(n+1)}{2} + n+1 -i \quad \ &\text{for  $1\leq i \leq n$}.
    \end{align*}
 Let $F_{\rho}$ be the irreducible, finite dimensional representation of $G$ of lowest weight $-\rho$ and denote by $\Tilde{F_{\rho}}$ the locally constant sheaf attached to $F_{\rho}$.
Then, for $n$ odd, we have
\begin{align*}
    \dim H_{cusp}^n (S(\Gamma(m_i)),\Tilde{F_{\rho}}) & \geq  d_\tau \mathrm{vol}(\Gamma(m_i) \backslash G) \\
&=\frac{h_{m_i}}{4^n} \prod_{i=1}^n \frac{\left| \prod_{j=i+1}^{n+1} (\tau_i - \tau_j) \right|}{(2\pi)^{i+1}} \\
& \cdot |D_\ell|^{\frac{(n-1)(n+2)}{4}} \cdot \prod_{k=2}^{n+1} Z_{\ell}(k)  \cdot \prod_{p \in T_\ell}\lambda_p.
\end{align*}
For $n$ even:
\begin{align*}
    \dim H_{cusp}^n (S(\Gamma(m_i)),\Tilde{F_{\rho}}) & \geq  d_\tau \mathrm{vol}(\Gamma(m_i) \backslash G) - C \frac{h_{m_i}}{m_i^n} \\
&=\frac{h_{m_i}}{4^n} \prod_{i=1}^n \frac{\left| \prod_{j=i+1}^{n+1} (\tau_i - \tau_j) \right|}{(2\pi)^{i+1}}  \\
&\cdot |D_\ell|^{\frac{n(n+3)}{4}} \left( \prod_{r=1}^n \frac{1}{(2\pi)^{r+1}}\right) \cdot \prod_{k=2}^{n+1} Z_{\ell}(k)  - C\frac{h_{m_i}}{m_i^n}.
\end{align*}
\end{thm}


\section{Application II: Rationality of L-functions} \label{rationality}

As before, let $\ell=\Q(\sqrt{-k})$ be a quadratic extension, $D_\ell$ the corresponding discriminant and $L_\ell$ the corresponding $L$-function. It is also the $L$-function corresponding to the Jacobi-symbol of $\ell$ interpreted as a character which we denote by $\chi_\ell$ (i.e. $L_\ell(s)=L(s,\chi_\ell)$). Then:

\begin{cor} \label{rational}
Let $\ell=\Q(\sqrt{-k})$ be an imaginary quadratic quadratic field. Let $\Gamma^{(1)}$ and $\Gamma^{(2)}$ be arithmetic subgroups as in Definition \ref{defgamma} of $SU(2n-1,1)$ and $SU(2n+1,1)$ respectively and let $\Gamma^{(i)}(m^{(i)})$ be net principal congruence subgroups (as in Definition \ref{congruence-subgroup}) of index $h_{m^{(i)}}$ in $\Gamma^{(i)}$, with $n\geq 2$ and $m^{(i)}\geq 3$. Let $\pi_{\tau^{(i)}} \in \hat{G}$ be the discrete series representation with Harish-Chandra parameter $\tau^{(i)}$, such that 
\begin{align*}
        \tau_i^{(1)} &= \tau_i^{(2)} \ for\ 1\leq i \leq 2n-1.,  \\
        \tau_1^{(1)} &> \tau_2^{(1)} > \ldots \tau_{2n-1}^{(1)}>\tau_{2n}^{(1)} +2n -1,\\
        \tau_{2n-1}^{(2)} &> \tau_{2n}^{(2)}> \tau_{2n+1}^{(2)} >\tau_{2n+2}^{(2)}+2n+1, 
        \end{align*}
where we set 
\begin{align*}
    \tau_{2n}^{(1)} &\ceq -\sum_{i=1}^{2n-1} \tau_i^{(1)}, \\
    \tau_{2n+2}^{(2)} &\ceq -\sum_{i=1}^{2n+1} \tau_i^{(2)}.
\end{align*}

This implies that $\pi_{\tau^{(i)}}$ is of integrable discrete series. Then
    \begin{align*}
    &\sqrt{|D_\ell|} \frac{L_\ell(2n+1)}{(2\pi)^{2n+1}} \\
      = &\frac{m(\Gamma^{(2)}(m^{(2)}),\pi_{\tau^{(2)}})}{m(\Gamma^{(1)}(m^{(1)}),\pi_{\tau^{(1)}})} \cdot \frac{32h_{m^{(1)}}}{h_{m^{(2)}}}  \frac{\displaystyle\prod_{\substack{1 \leq i <  2n \\ }} (\tau_i^{(1)}-\tau_{2n}^{(1)})}{\displaystyle \prod_{\substack{1 \leq i < j \leq 2n+2 \\ 2n \leq j}}(\tau_i^{(2)}-\tau_j^{(2)})}\\
      & \cdot \prod_{p \in T_\ell^{(1)}}\frac{p^{2n}-1}{p+1} \cdot \prod_{p \in T_\ell^{(2)}}\frac{p+1}{p^{2n+2}-1} |D_\ell|^{2n}  \frac{(2n+2)!}{(-1)^n B_{2n+2}},
    \end{align*}
    where the set $T_\ell^{(1)}$ (resp. $T_\ell^{(2)}$) consists of those ramified primes $p=\mathfrak{p}^2$ in $\ell$, where $(-1)^{n-1} \not\in N_{\ell_{\mathfrak{p}} / \Q_p}(\ell_{\mathfrak{p}}^\times)$ (resp. $(-1)^{n} \not\in N_{\ell_{\mathfrak{p}} / \Q_p}(\ell_{\mathfrak{p}}^\times)$).
In particular, this implies that $\sqrt{|D_\ell|} \frac{L_\ell(2n+1)}{(2\pi)^{2n+1}}$ is a rational number.
\end{cor}
\begin{proof}
From Theorem \ref{mult} for the two cases $n \mapsto 2n-1$ and $n \mapsto 2n+1$, we can deduce (by dividing one equation by the other) that 
    \begin{align*}
        & \frac{m(\Gamma^{(2)}(m^{(2)}),\pi_{\tau^{(2)}})}{m(\Gamma^{(1)}(m^{(1)}),\pi_{\tau^{(1)}})}\\
        =  &\frac{h_{m^{(2)}}}{16h_{m^{(1)}}}  
        \frac{\displaystyle \prod_{\substack{1 \leq i < j \leq 2n+2 \\ 2n \leq j}}(\tau_i^{(2)}-\tau_j^{(2)})}{\displaystyle\prod_{\substack{1 \leq i <  2n \\ }} (\tau_i^{(1)}-\tau_{2n}^{(1)})}
    \cdot |D_\ell|^{\frac{4n+1}{2}} \\
    &\cdot \frac{L_\ell(2n+1)}{(2\pi)^{2n+1}}\cdot \frac{\zeta(2n+2) }{(2\pi)^{2n+2}}
    \cdot \prod_{p \in T_\ell^{(1)}}\frac{1}{\lambda_p^{(1)}} \cdot \prod_{p \in T_\ell^{(2)}}\lambda_p^{(2)}.
    \end{align*}
    We rearrange this to obtain the expression
    \begin{align*}
      &|D_\ell|^{\frac{1}{2}}  \frac{L_\ell(2n+1)}{(2\pi)^{2n+1}} \\
      = &\frac{m(\Gamma^{(2)}(m^{(2)}),\pi_{\tau^{(2)}})}{m(\Gamma^{(1)}(m^{(1)}),\pi_{\tau^{(1)}})} \cdot \frac{16 h_{m^{(1)}}}{h_{m^{(2)}}} \\
      & \cdot \frac{\displaystyle\prod_{\substack{1 \leq i <  2n \\ }} (\tau_i^{(1)}-\tau_{2n}^{(1)})}{\displaystyle \prod_{\substack{1 \leq i < j \leq 2n+2 \\ 2n \leq j}}(\tau_i^{(2)}-\tau_j^{(2)})} \prod_{p \in T_\ell^{(1)}}\lambda_p^{(1)}\cdot \prod_{p \in T_\ell^{(2)}}\frac{1}{\lambda_p^{(2)}}|D_\ell|^{2n}  \frac{(2\pi)^{2n+2}}{\zeta(2n+2) } 
    \end{align*}
    It is well known that 
    $$
    \zeta(2n+2)=(-1)^n \frac{(2\pi)^{2n+2}B_{2n+2}}{2(2n+2)!},
    $$
    where $B_n$ denotes the $n$-th Bernoulli number. 
   We can deduce that
    \begin{align*}
    &|D_\ell|^{\frac{1}{2}}  \frac{L_\ell(2n+1)}{(2\pi)^{2n+1}} \\
      = &\frac{m(\Gamma^{(2)}(m^{(2)}),\pi_{\tau^{(2)}})}{m(\Gamma^{(1)}(m^{(1)}),\pi_{\tau^{(1)}})} \cdot \frac{32h_{m^{(1)}}}{h_{m^{(2)}}}  \frac{\displaystyle\prod_{\substack{1 \leq i <  2n \\ }} (\tau_i^{(1)}-\tau_{2n}^{(1)})}{\displaystyle \prod_{\substack{1 \leq i < j \leq 2n+2 \\ 2n \leq j}}(\tau_i^{(2)}-\tau_j^{(2)})}\\
      & \cdot \prod_{p \in T_\ell^{(1)}}\lambda_p^{(1)}\cdot \prod_{p \in T_\ell^{(2)}}\frac{1}{\lambda_p^{(2)}} |D_\ell|^{2n}  \frac{(2n+2)!}{(-1)^n B_{2n+2}}.
    \end{align*}
    The expression on the right hand side only involves rational numbers, which implies that
    \begin{align} \label{rat}
        \sqrt{|D_\ell|} \frac{L_\ell(2n+1)}{(2\pi)^{2n+1}}\in \Q.
    \end{align}

\end{proof}

\begin{rmk}
    A possible choice for $\tau^{(i)}$ is for example
    \begin{align*}
        \tau_i^{(1)} & \ceq 2n+2-i &\ \text{for }i\leq 2n-1, \\
        \tau_i^{(2)} &\ceq 2n+2 -i  &\ \text{for }i\leq 2n+1. 
    \end{align*}
    in which case we have 
    \begin{align*}
        \frac{\displaystyle\prod_{\substack{1 \leq i <  2n \\ }} (\tau_i^{(1)}-\tau_j^{(1)})}{\displaystyle \prod_{\substack{1 \leq i < j \leq 2n+2 \\ 2n \leq j}}(\tau_i^{(2)}-\tau_j^{(2)})} = \frac{\displaystyle\prod_{\substack{1 \leq i <  2n \\ }} (2n-i+2+(n+2)(2n-1))}{\displaystyle (2n-1)! (2n)!\prod_{\substack{1 \leq i \leq 2n+1}}(2n+2-i+(n+1)(2n+1))} \\
        =  \frac{1}{(2n-1)! (2n)!(2n^2+3n+2)} \prod_{\substack{1 \leq i <  2n \\ }} \frac{2n^2+5n-i}{2n^2+5n+3-i}. 
    \end{align*}
\end{rmk}